\documentclass{article}
\RequirePackage[numbers, sort&compress]{natbib}

\usepackage{graphicx}
\usepackage{blindtext}
\usepackage{amsmath}
\usepackage{amsfonts}
\usepackage{amsthm}
\usepackage{amssymb}
\usepackage{tabularx}
\usepackage{mathtools}
\usepackage{epstopdf}
\usepackage{pdfpages}
\usepackage[margin=1in]{caption}
\usepackage{listings}
\usepackage{import}
\usepackage{framed}

\usepackage[utf8]{inputenc}

\newtheorem{lemma}{Lemma}[section]
\newtheorem{theorem}{Theorem}[section]

\newtheorem{definition}{Definition}

\theoremstyle{remark}

\usepackage[nottoc,notlot,notlof]{tocbibind}
\usepackage{algorithmic}

\title{On tail dependence parameters for non-continuous and autocorrelated margins}
\author{Victory Idowu}
\date{January 2025}

\begin{document}
\maketitle
\begin{abstract}
Tail dependence plays an essential role in the characterization of joint extreme events in multivariate data. However, most standard tail dependence parameters assume continuous margins. 
This note presents a form of tail dependence suitable for non-continuous and discrete margins. We derive a representation of tail dependence based on the volume of a copula and prove its properties. We utilize a bivariate regular variation to show that our new metric is consistent with the standard tail dependence parameters on continuous margins. We further define tail dependence on autocorrelated margins where the tail dependence parameter examine lagged correlation on the sample.
\end{abstract}

\section{Introduction}

\label{sec:intro}

Tail dependence is a measure that quantifies the likelihood of simultaneous extreme values in multiple variables \cite{ledford1997modelling}. It is an essential measure for the modeling of extreme events in finance and insurance. As an emerging area of research, the most commonly used metrics for tail dependence are the standard upper tail dependence $\lambda_U$ and lower tail dependence $\lambda_L$ parameters. $\lambda_U$, $\lambda_L$ quantify the asymptotic probability of one variable being in upper and lower quantiles of the distribution respectively \cite{nelsen2006introduction, joe1997multivariate}. These scalar parameters are implicitly defined for continuous margins \cite{joe1997multivariate}, or the stronger assumption of a Fr\'{e}chet distributions \cite{ledford1997modelling, joe1997multivariate}. However, they are limited in capturing the geometric nature of tail dependence in the asymptotic limit as they assume identical scaling in all variables \citet{wadsworth2024statistical}. Multivariate data tail dependence often varies directionally, as the presence of heavy tails may be dominated by different variables in different parts of the extremes of the distribution \cite{ledford1997modelling, draisma2004bivariate}.

Furthermore, uniform scaling may not occur when margins are non-continuous. Non-continuous margins may have regions of discontinuity in the limit and potentially preventing the limit. Uniform scaling in standard tail dependence parameters reflect the implicit assumption of regular variation in all variables through a scaling parameter \cite{chen2025heavy, chen2019bivariate}. Through the copula representation of multivariate data it becomes apparent that assuming both marginals exhibit the same scaling rate oversimplifies the dependency structure \cite{li2018operator, hu2022tail}. Note that not all copulas with tail dependence show regular variation. For instance, the bivariate Student's $t$ copula shows tail dependence at all degrees of freedom \citep{schmidt2002tail}.

Moreover, despite standard tail dependence representation with copulas, there has been limited discussion on the geometric nature of dependence in the extremes from the perspective of asymptotic limits. Operator tail dependence measures were introduced to account for direction-dependent tail behavior \citet{li2018operator}. However, they make explicit directional assumptions with limited geometric interpretation in non-continuous margins. Both approaches are yet to consider the non-uniformity of convergence due to discrete and discontinuities \cite{wadsworth2024statistical,li2024tail}. These limitations highlight the need for a geometric representation of tail dependence that can handle different data types. 

We propose a geometric approach that connects tail dependence with the volume of the copula, without having to assume continuous margins. The volume based approach enables for a richer geometric interpretation of how dependence can evolve over shrinking tail regions. In particular, our approach has the interpretation of quantifying tail dependence as the limit over the continuous ranges of the copula's domain. 

This paper proceeds as follows. In Section \ref{sec:c1b-copulatheory-taildependence}, we provide a brief overview of key concepts from copula theory and tail dependence. In Section \ref{sec:c1b-generalizedtail} we define a new representation of the tail dependence for bivariate copulas applicable to non-continuous margins. We prove that this definition corresponds with the traditional definition of tail dependence. 

\section{Preliminaries} \label{sec:c1b-copulatheory-taildependence}
Throughout this paper, we will show all tail dependence concepts through the bivariate copula representation. Recall, the definition of a copula with 2 marginals \cite{nelsen2006introduction}:
\begin{definition}[Bivariate Copula]
$C:[0,1]^2 \rightarrow [0,1]$ is a bivariate copula if for every $u, v \in [0,1]$,
\begin{align*}
    C(0, u) =  0 = C(u, 0), \; C(1, u) =  u = C(u, 1),\; C(u,v)  = C(v,u)
\end{align*}
 For any $u_1 \leq v_1, u_2 \leq v_2$, $C$ has the two increasing property,
 \begin{equation*}
     C(u_2, v_2) - C(u_2, v_1) - C(u_1, v_2) + C(u_1, v_1) \geq 0. 
 \end{equation*}  
\end{definition}

One of the ways we can define a function measure of length on $C$ is to use its 2-increasing property, which is non-negative at all points in $C$'s domain. It is called the $C$-volume. More formally, 
\begin{definition}[C-volume, $V_C$ \citep{nelsen2006introduction}]
For any $C:[0,1]^2 \rightarrow [0,1]$ and any $u_1 \leq v_1, u_2 \leq v_2$, its $C$-volume is 
\begin{equation*}
    V_C([u_1, u_2] \times [v_1, v_2]) = C(u_2, v_2) - C(u_2, v_1) - C(u_1, v_2) + C(u_1, v_1).
\end{equation*}
\end{definition}

Clearly \citep{nelsen2006introduction}, $V_C ([0,u] \times [0,v]) = C(u,v)$. The uniqueness of the copula for continuous distributions is established by Sklar's theorem \citep{sklar1959fonctions}. Existence is only guaranteed for non-continuous d.fs \citep{joe2014dependence}.

Consider the standard formulation for tail dependence, in which the margins are assumed to be continuous.

For any bivariate distribution, the tail dependence is defined by conditional distributions at the upper and lower quadrant of its joint distribution. The upper tail dependence parameter $\lambda_U$ \cite{nelsen2006introduction}, is the limit,
\begin{align}
\lambda_U :=  \lim_{t \rightarrow 1^-}  \mathbb{P} \left( Y > G^{(-1)} (t) | X > F^{(-1)} (t) \right) = \lim_{t \rightarrow 1^-}  2 - \frac{1-C(t,t)}{1-t}.
\end{align}

Its lower tail dependence $\lambda_L$, is the limit, 
\begin{align}
    \lambda_L :=  \lim_{t \rightarrow 0^+}  \mathbb{P} \left( Y \leq G^{(-1)} (t) | X \leq F^{(-1)} (t) \right)  = \lim_{t \rightarrow 0^+}  \frac{C(t,t)}{t} . 
\end{align}

Tail dependence over a single sample set in which autocorrelation may be present is through the auto tail dependence function \citep{reiss1997statistical}. We vary the definition in \citet{reiss1997statistical} assuming only right continuity in the margins. 
\begin{definition}[Auto tail dependence parameter \citep{reiss1997statistical}]
Let $X_1, \ldots X_n$ be a series of i.i.d random variables with cdf $F$. For $ i \leq n -h$, the auto-tail-dependence function at level $t$ is,
\begin{align*}
    \lambda^t_{X_{i+h} | X_{i}} :=  \mathbb{P}( X_{i+h} > F^{-1} (t) | X_{i} > F^{-1} (t) ) = \mathbb{P}( X_{1+h} > F^{-1} (t) | X_{1} > F^{-1} (t) ).
\end{align*}
The auto tail dependence parameter is,
\begin{equation*}
    \lambda_{X_{i+h} | X_{i}}  = \lim_{t \rightarrow 1^{-} } \lambda^t_{X_{i+h} | X_{i}} \; . 
\end{equation*}
\end{definition}
Note that the auto tail dependence parameter is the upper tail dependence parameter evaluated over the same sample. Also, note that a similar measure for using lower rail-dependence can be defined, see \citep{charpentier2007lower}. 

First, we introduce the concept of bivariate regular variation \citep{chen2019bivariate}. The type of bivariate regular variation is classified separately depending on whether the margins are different or the same. Bivariate regular variation is defined in terms of the upper tails of the joint distribution \citep{chen2019bivariate}. 

\begin{definition}[Non-standard Bivariate Regular Variation \citep{chen2019bivariate}] \label{def:c1b-non-standardBRV}
Let $X,Y$ be a non-negative pair of random variables with right-continuous cdf $F_x, F_Y$. If there exists, $\nu$ non-degenerate defined on $\mathbb{R}_{+}^2$, for every Borel set $B \subset \mathbb{R}_{+}^2$ bounded away from $(0,0)$ and $\nu$-continuous. $(X,Y)$, follow a non-standard regular variation structure if the following limit relation exists,
\begin{equation}
    \lim_{x \rightarrow \infty} x \mathbb{P} \left(  \left( \frac{X}{U_X (x)}, \frac{Y}{U_Y (x)}\right) \in B \right) = \nu(B)
\end{equation}
where for the random variable $X$ (similarly $Y$) $\bar{F}(x) = 1 - F(X)$ and
$    U_X (x) = \inf \left\{ y \in \mathbb{R} | \bar{F}(y) \leq 1 / x \right\} , x > 0$.
\end{definition}

When $X,Y$ have the same distribution function, we can define the simplified standard Bivariate Regular Variation \citep{chen2019bivariate, chen2025heavy}. Continuing from Definition (\ref{def:c1b-non-standardBRV}), let $X,Y$ be a non-negative pair of random variables with right-continuous cdf $F$ and $U = U_X = U_Y$. We say there is standard bivariate variation if the following limit exists,
\begin{equation}
    \lim_{x \rightarrow \infty} x \mathbb{P} \left( \frac{( X, Y)}{U (x)} \in B \right) = \nu(B).
\end{equation}

The current definitions for  $\lambda_U, \lambda_L$, are able to characterize tail dependence under the assumption of convergence to limit without jumps. This occurs when $C$ has continuous margins. However, the definition of tail dependence is not well-defined when $C$ has jumps. In particular, non-continuous or mixed margins may mean that $C$ is not smooth near its upper or lower tail dependence limit. 

\section{Generalized Tail Dependence}\label{sec:c1b-generalizedtail}
The current definitions of $\lambda_U, \lambda_L$ and $\lambda_{X_{i+h} | X_{i}}$  tail dependence cannot be directly applied to non-continuous margins as the limit may not be well-defined. Due to this, we define a generalized tail dependence based on the volume of the copula and prove it is well-defined. Afterwards, we show that it is equivalent to the traditional definition of tail dependence. 

Given above, we define a form of tail dependence suitable for non-continuous, discrete and mixed margins utilizing the volume of a copula. We term these new dependence measures, \textit{generalized tail dependence}. 
\begin{definition}[Generalized Tail Dependence]
Let $C$ be a bivariate copula and $V_C$ be its corresponding volume. 

The upper tail dependence $\Tilde{\lambda}_U$ is, 
\begin{equation}
    \Tilde{\lambda}_U := \lim_{t \rightarrow 1^-} \frac{V_C ([t,1] \times [t,1]) }{ 1- t}
\end{equation}

The lower tail dependence $\Tilde{\lambda}_L$ is,
\begin{equation}
    \Tilde{\lambda}_L := \lim_{t \rightarrow 0^+}  \frac{V_C( [0,t] \times [0,t]) }{t}
\end{equation}
\end{definition}

\begin{lemma}
When $C$ is bivariate with continuous margins, $\Tilde{\lambda}_U$ is equivalent to $\lambda_U$. Likewise, $\Tilde{\lambda}_L$ is equivalent to $\lambda_L$. 
\end{lemma}

\begin{proof}
 Proof is immediate from the definition of the $V_C$. Note that for continuous margins, $V_C ( [0,t] \times [0,t] ) = C(t,t)$ and similarly, $ V_C( [t,1] \times [t,1]) = C(1,1) - C(t,1) - C(1,t) + C(t,t)$. From the definition of a copula, $C(1,1) = 1$ and $C(t,1) = C(1,t) = t$. Hence,$ V_C( [t,1] \times [t,1]) =  1 - 2t + C(t,t)$ from which the result is immediate. 
\end{proof}

\begin{theorem}\label{def:c1b-bivaraite-defined1}
For any bivariate copula $C$, with valid discrete, mixed or non-continuous margins, if $\Tilde{\lambda}_U$, $\Tilde{\lambda}_L$ exist then they are well defined. 
\end{theorem}

\begin{proof}
We first prove the statement for lower tail dependence. Suppose $C$ has discontinuities over the box $B = [0,t] \times [0,t]$. Consider the partition of $B = \cup_{k,j} B_{k,j}$, where $C$ is continuous. Define, 
\begin{equation*}
B_{k,j} = [u_k, u_{k+1}] \times [v_j, v_{j+1}], \qquad k = 0, \ldots, m-1, \; j=0, \ldots, n-1,
\end{equation*}
where $0 = u_0 < u_1 < \ldots u_m = t$ and $0 = v_0 < v_1 < \ldots < v_p = t$. For each $B_{k,j}$, 
\begin{equation*}
  V_C([u_k, u_{k+1}] \times [v_j, v_{j+1}]) =  C(u_{k+1}, v_{j+1}) - C(u_k, v_{j+1}) - C(u_{k+1}, v_j)  + C(u_k, v_{j})
\end{equation*}
Therefore the total volume over $[0,t] \times [0,t]$ is the sum over $B_{k,j}$, 
\begin{equation*}
    V_C( [0,t] \times [0,t]) = \sum_{k=0}^{m-1} \sum_{j=0}^{p-1}  V_C(B_{k,j}).
\end{equation*}

As $t \rightarrow 0^+$, the partition $B_{k,j}$ becomes finer and finer, which is well-defined as,
\begin{equation*}
    \frac{V_C( [0,t] \times [0,t]) }{t } = \frac{1}{t} \sum_{k=0}^{m} \sum_{j=0}^{p}  V_C(B_{k,j}).
\end{equation*}

Similarly for upper tail dependence, consider the discontinuities over the box $B^\prime = [t,1] \times [t,1]$. Consider the partition of $B^\prime  = \cup_{k,j} B_{k,j}$, where $C$ is continuous. Define, 
\begin{equation*}
B_{k,j}^\prime  = [u_k, u_{k+1}] \times [v_j, v_{j+1}], \qquad k = 0, \ldots, m-1, \; j=0, \ldots, n-1,
\end{equation*}
where $t = u_0 < u_1 < \ldots u_m = 1$ and $t = v_0 < v_1 < \ldots < v_p = 1$. The total volume of $C$ over $[t,1] \times [t,1]$ is the sum,
\begin{equation*}
    V_C( [t,1] \times [t,1]) = \sum_{k=0}^{m-1} \sum_{j=0}^{p-1}  V_C(B_{k,j}^\prime)
\end{equation*}
Thus the limit is well-defined as $t \rightarrow 1^-$. Then, 
\begin{equation*}
    \frac{V_C( [t,1] \times [t,1]) }{ 1 -t} = \frac{1}{1-t} \sum_{k=0}^{m-1} \sum_{j=0}^{p-1}  V_C(B_{k,j}^\prime).
\end{equation*}
\end{proof}

Likewise we can derive a generalized auto tail dependence parameter can be defined and show it is well defined. 
\begin{definition}[Generalized Autocorrelation Tail Dependence]
Let $X_1, \ldots X_n$ be a series of identically distributed but not necessarily independent random variables with cdf $F$. $F$ can be continuous or non-continuous. Let $\Tilde{C}$ be a bivariate copula defined on the joint distribution of $F$ on itself and $V_{\Tilde{C}}$ be its corresponding volume. 
 \begin{equation*}
       \Tilde{\lambda}_{X_{i+h} | X_{i}} = \lim_{t \rightarrow 1^-} \frac{V_{\Tilde{C}} ([t,1] \times [t,1]) }{ 1- t}
 \end{equation*}
\end{definition}

\begin{lemma}\label{lemma:c1b-autocorr}
$ \Tilde{\lambda}_{X_{i+h} | X_{i}}$ is well defined.    
\end{lemma}
\begin{proof}
   Proof is immediate from Theorem \ref{def:c1b-bivaraite-defined1} by setting $\Tilde{C} = C$. 
\end{proof}
We can show the equivalence of $\Tilde{\lambda}_U$ in terms of regular variation on $V_C$. 
\begin{theorem}
For arbitrary $w,z \in \mathbb{R}^+$, let $B = [w, \infty) \times [z, \infty)$.
Let $C$ be a bivariate copula, if the limit exists, 
\begin{equation}
    \nu(w,z) := \lim_{x \rightarrow \infty} x \left[ 1 - \frac{z}{x} - \frac{w}{x}  + V_C \left([1 - \frac{z}{x}, 1] \times [1-\frac{w}{x},1] \right) \right]
\end{equation}
then $C$ has regular variation and upper tail dependence and that limit is $\Tilde{\lambda}_U = \nu (1,1)$.
\end{theorem}

\begin{proof}
To determine if there is regular variation in terms of Definition \ref{def:c1b-non-standardBRV} must be defined. Hence, we want to know the limit of $ \mathbb{P} \left( \frac{X}{U_X (x)} > Z, \frac{Y}{U_Y(x)} > W \right) $ as $x \rightarrow \infty$. Note that for any copula $C$, $P(U  >u , V > u) = 1 - u -v - C(1-u,1-v)$ \citep{nelsen2006introduction}. Consider the partition $B_{k,j}^{\prime \prime}$, where $C$ is continuous in $[1 - \frac{z}{x}, 1] \times [1-\frac{w}{x},1]$. Define, 
\begin{equation*}
B_{k,j}^{\prime \prime} = [u_k, u_{k+1}] \times [v_j, v_{j+1}], \qquad k = 0, \ldots, m-1, \; j=0, \ldots, n-1,
\end{equation*}
where $1 - \frac{z}{x} = u_0 < u_1 < \ldots u_m = 1$ and $1-\frac{w}{x} = v_0 < v_1 < \ldots < v_p = 1$. The total volume is defined as,
\begin{equation}
    V_C \left([1 - \frac{z}{x}, 1] \times [1-\frac{w}{x},1] \right) := \sum_{k=0}^{m-1} \sum_{j=0}^{p-1}  V_C(B_{k,j}^{\prime \prime})
\end{equation}
Therefore, 
\begin{equation} \label{eq:c1b-prob1}
   \mathbb{P} \left( \frac{X}{U_X (x)} > Z, \frac{Y}{U_Y(x)} > W \right)  = 1 - \frac{z}{x} - \frac{w}{x}  + V_C \left([1 - \frac{z}{x}, 1] \times [1-\frac{w}{x},1] \right). 
\end{equation}
Equivalence to $\lambda_U$ holds by the substitution $t = 1 - 1/x$.
\end{proof}

\end{document}